\documentclass{amsart}
\usepackage{verbatim,amssymb,amsmath,amscd,latexsym,amsbsy,mathrsfs}
\usepackage{enumerate}
 \input{xy}
\xyoption{all}

\newtheorem{thm}{Theorem}[section] \newtheorem{pro}[thm]{Proposition}
\newtheorem{lemma}[thm]{Lemma}
\newtheorem{cor}[thm]{Corollary}
\numberwithin{equation}{section}

\theoremstyle{remark}

\theoremstyle{definition} 
\newtheorem{rmk}[thm]{Remark} \newtheorem*{clm}{Claim}

\DeclareMathAlphabet{\mathpzc}{OT1}{pzc}{m}{it}
\DeclareMathOperator*{\spec}{Spec} 
 
 \DeclareMathOperator*{\QF}{QF}
\DeclareMathOperator*{\Gal}{Gal}

\newcommand{\ZZ}{\mathbb{Z}} \newcommand{\Aff}{\mathbb{A}}
\newcommand{\PP}{\mathbb{P}} \newcommand{\FF}{\mathbb{F}}

\newcommand{\cO}{\mathcal{O}} 
\newcommand{\scrD}{\mathcal{D}}

\begin{document}
\title{Killing Wild ramification} 
 \author{
   Manish Kumar
  } \thanks{The author is supported by SFB/TR-45 grant}
  \address{
Department of Mathematics \\
Universit\"at Duisburg-Essen \\
45117 Essen, Germany
  }
  \email{manish.kumar@uni-due.de}
 \begin{abstract}
  We compute the inertia group of the compositum of wildly ramified Galois covers. It is used to show that even the $p$-part of the inertia group of a Galois cover of $\PP^1$ branched only at infinity can be reduced if there is a jump in the ramification filtration at two (in the lower numbering) and certain linear disjointness statement holds.
 \end{abstract}
\maketitle

\section{Introduction}
 
Let $k$ be a field of characteristic $p$. Let $\phi:X\to Y$ be a finite Galois $G$-cover of regular irreducible $k$-curves branched at $\tau\in Y$. Let $I$ be the inertia subgroup of $G$ at a point of $X$ above $\tau$. It is well known, $I=P\rtimes \mu_n$ where $P$ is a $p$-group, $\mu_n$ is a cyclic group of order $n$ and $(n,p)=1$. 
Abhyankar's lemma can be viewed as a tool to modify the tame part of the inertia group. 
For instance, suppose $k$ contains $n^{th}$-roots of unity. Let $y$ be a regular local parameter of $Y$ at $\tau$. Let $Z\to Y$ be the Kummer cover of regular curves given by the field extension $k(Y)[y^{1/n}]/k(Y)$ and $\tau'\in Z$ be the unique point lying above $\tau$. Then the pullback of the cover $X\to Y$ to $Z$ is a Galois cover of $Z$ branched at $\tau'$. But the inertia group at any point above $\tau'$ is $P$. A wild analogue of this phenomenon appears as Theorem \ref{killing-inertia}. 

Assume $k$ is also algebraically closed field and let $X\to \PP^1$ be a Galois $G$-cover of $k$-curves branched only at $\infty$. Let $I$ be the inertia subgroup at some point above $\infty$ and $P$ be the sylow-$p$ subgroup of $I$. Then noting that the tame fundamental group of $\Aff^1$ is trivial, it can be seen that the conjugates of $P$ in $G$ generate the whole of $G$. Abhyankar's inertia conjecture states that the converse should also be true. More precisely, any subgroup of a quasi-$p$ group $G$ of the form $P\rtimes \mu_n$ where $P$ is a $p$-group and $(n,p)=1$ such that conjugates of $P$ generate $G$ is the inertia group of a $G$-cover of $\PP^1$ branched only at $\infty$. 

An immediate consequence of a result of Harbater (\cite[Theorem 2]{Harbater-adding-branch-points}) shows that the inertia conjecture is true for every sylow-$p$ subgroup of $G$. In fact Harbater's result shows that if a $p$-subgroup $P$ of $G$ occurs as the inertia group of a $G$-cover of $\PP^1$ branched only at $\infty$ and $Q$ is a $p$-subgroup of $G$ containing $P$ then there exists a $G$-cover of $\PP^1$ branched only at $\infty$ so that the inertia group is $Q$. Proposition \ref{killing.wild} and a study of wild ramification filtration (Proposition \ref{characterize.G_2}) enables us to show that in certain cases the given $G$-cover of $\PP^1$ can be modified to obtain a $G$-cover of $\PP^1$ branched only at $\infty$ so that the inertia group of this new cover is smaller than the inertia group $P$ of the original cover (Theorem \ref{reducing-inertia}).

So far the inertia conjecture is only known for some explicit groups. See for instance \cite[Theorem 5]{Bouw-Pries} and \cite[Theorem 1.1]{Muskat-Pries}.

\section{Filtration on ramification group}

For a complete discrete valuation ring (DVR) $R$, $v_R$ will denote the valuation associated to $R$ with the value group $\ZZ$. Let $S/R$ be a finite extension of complete DVRs such that $\QF(S)/\QF(R)$ is a Galois extension with Galois group $G$. Let us define a decreasing filtration on $G$ by 
$$G_i=\{\sigma\in G: v_S(\sigma x -x)\ge i+1, \, \forall x\in S\}$$
Note that $G_{-1}=G$ and $G_0$ is the inertia subgroup. This filtration is called the ramification filtration. For every $i$, $G_i$ is a normal subgroup of $G$. The following are some well-known results.

\begin{pro}\cite[IV, 1, Proposition 2 and 3]{Serre-local.fields}\label{filtration}
 Let $S/R$ be a finite extension of complete DVRs such that $\Gal(\QF(S)/\QF(R))=G$. Let $H$ be a subgroup $G$. Let $K$ be the fixed subfield of $\QF(S)$ under the action of $H$. Let $T$ be the normalization of $R$ in $K$. Then $T$ is a complete DVR, $\Gal(\QF(S)/K)=H$ and the ramification filtration on $H$ is induced from that of $G$, i.e. $H_i=G_i \cap H$. Moreover, if $H=G_j$ for some $j\ge 0$ then $(G/H)_i=G_i/H$ for $i\le j$ and $(G/H)_i=\{e\}$ for $i\ge j$.
\end{pro}

\begin{pro}\cite[IV, 2, Corollary 2 and 3]{Serre-local.fields}\label{higher.ramification.group}
 The quotient group $G_0/G_1$ is a prime-to-$p$ cyclic group and if the residue field has characteristic $p>0$ then for $i\ge 1$, $G_i/G_{i+1}$ is an elementary abelian group of exponent $p$. In particular $G_1$ is a $p$-group.
\end{pro}

\begin{lemma}\label{Hilbert-transitivity}
 Let $S/R$ be an extension of DVRs such that $\QF(S)/\QF(R)$ is Galois with $\Gal(\QF(S)/\QF(R))=G$. Let $H$ be a normal subgroup of $G$ and $T$ be the normalization of $R$ in $\QF(S)^{H}$ then 
 $$\sum_{i=0}^{\infty}(|G_i|-1)=e_{S/T}\sum_{i=0}^{\infty}(|(G/H)_i|-1) + \sum_{i=0}^{\infty}(|H_i|-1)$$
\end{lemma}
\begin{proof}
 This follows from the transitivity of the different $\scrD_{S/R}=\scrD_{S/T}\scrD_{T/R}$ \cite[III, 4, Proposition 8]{Serre-local.fields}, Hilbert's different formula $d_{S/R}=\sum_{i=0}^{\infty}(|G_i|-1)$ (\cite[Theorem 3.8.7]{book-stichtenoth}) and $v_S(x)=e_{S/T}v_T(x)$ for $x\in \QF(T)$.
\end{proof}

\begin{lemma}\label{lemma.different}
 Let $S/R$ be a totally ramified extension of complete DVRs over a perfect field $k$ of characteristic $p>0$. Suppose $\QF(S)$ is generated over $\QF(R)$ by $\alpha\in \QF(S)$ with $\alpha^p-\alpha\in \QF(R)$ and $v_R(\alpha^p-\alpha)=-1$. Then the degree of the different $d_{S/R}=2|G|-2$. 
\end{lemma}

\begin{proof}
 Note that since $S/R$ is totally ramified, their residue fields are same and by \cite{cohen} the residue field is isomorphic to the field of coefficient of $R$ and $S$. Replacing $k$ by this residue field we may assume that the residue fields of $S$ and $R$ are $k$.

 We know that $|G|=p^l$ for some $l\ge 0$. We will prove the lemma by induction on $l$. If $l=0$ then the statement is trivial. Suppose $l=1$. Then by hypothesis there exists $\alpha \in \QF(S)$ with $\alpha^p-\alpha \in R$ and $v_R(\alpha^p-\alpha)=-1$. Let $x=(\alpha^p-\alpha)^{-1}$ and $y=\alpha^{-1}$ then $v_S(x)=e_{S/R}v_R(x)=p$ and $v_S(y)=1$. By Cohen structure theorem $R=k[[x]]$ and $S=k[[y]]$. Also we have that $m(y)=0$ where $m(T)=T^p+xT^{p-1}-x \in R[T]$. So $m(T)$ is a minimal polynomial of $y$ over $\QF(R)$. By \cite[III, 6, Corollary 2]{Serre-local.fields}, $d_{S/R}=v_S(m'(y))$. But $m'(y)=-xy^{p-2}$. So $d_{S/R}=v_S(x)+(p-2)v_S(y)=2p-2$.

 Now in general assume $l\ge 1$. Note that $G=(\ZZ/p\ZZ)^l$, so by hypothesis there exist $\alpha_1,\ldots \alpha_l\in \QF(S)$ such that
 \begin{enumerate}
  \item $\alpha_i^p-\alpha_i=u_ix^{-1}$ for some units $u_i\in R$ and 
  \item $\QF(R)(\alpha_i)$ and $L_{i-1}=\QF(R)(\alpha_j| 1\le j < i)$ are linearly disjoint over $\QF(R)$ for $1 \le i \le l$.
 \end{enumerate}

 Note that $L_0=\QF(R)$ and $L_l=\QF(S)$. Let $T_i$ be the normalization of $R$ in $L_i$ for $0\le i\le l$. Let $y_0=x$. For simplicity, let $v_i$ denote the valuation $v_{T_i}$ for $0\le i\le l$. Since $S/R$ is totally ramified, so is $T_i/T_{i-1}$ for $1\le i \le l$. Hence $e_{T_i/T_{i-1}}=p$. Note that $v_0(y_0)=1$.

 \begin{clm}
   For each $0\le i\le l-1$ and $i < j \le l$, there exist $\beta_{i,j}\in \QF(S)$ such that the following holds
   \begin{enumerate}
     \item $\beta_{i,j}^p-\beta_{i,j}\in L_i$,
     \item $v_i(\beta_{i,j}^p-\beta_{i,j})=-1$,
     \item $L_i(\beta_{i,j}; i < j \le n)=L_n$ for $i<n\le l-1$ 
     \item $v_{i+1}(\beta_{i,i+1})=-1$
   \end{enumerate}
   We define $y_{i+1}=\beta_{i,i+1}^{-1}$.
 \end{clm}
 \begin{proof}[Proof of the claim]
   We shall proof this by induction. For $i=0$, we take $\beta_{0,j}=\alpha_j$. The first and the second statement is same as the hypothesis of the lemma. The third statement follows from the definition of $L_n$'s. For the fourth statement note that $\beta_{0,1}=\alpha_1$. Since $v_1(\alpha_1)< 0$, we have $v_1(\alpha_1^p)= v_1(\alpha_1^p-\alpha_1)=v_1(x^{-1})$. So $v_1(\alpha_1)=p^{-1}v_1(x^{-1})=p^{-1}pv_0(x^{-1})=-1$.

   Suppose the claim is true for a fixed $i\ge 0$ and $i<l-1$. Then we have $\beta_{i,j}\in \QF(S)$ for $i < j \le l$ satisfying the four properties listed in the claim. Also note that $v_i(y_i)=1$. So $T_i=k[[y_i]]$. Hence we can write explicitly $\beta_{i,j}^p-\beta_{i,j}=c_jy_i^{-1}+d_j+f_j(y_i)$ where $c_j,d_j\in k$, $c_j\ne 0$ and $f_j(y_i)\in T_i$ has order at least 1. Let $g_j=f_j+f_j^p+f_j^{p^2}+\ldots \in T_i$ then $g_j-g_j^p=f_j$. Let $\gamma_{i,j}=\beta_{i,j}-g_j$. Then $\gamma_{i,j}$ also satisfies the four properties of the claim. Moreover $\gamma_{i,j}^p-\gamma_{i,j}=c_jy_i^{-1}+d_j$. Hence replacing $\beta_{i,j}$ by $\gamma_{i,j}$, we may assume 
   \begin{equation}
    \beta_{i,j}^p-\beta_{i,j}=c_jy_i^{-1}+d_j
   \end{equation}

   Now for any $j$ such that $i+1< j \le l$. We define $\beta_{i+1,j}=\beta_{i,j}-a_j\beta_{i,i+1}$ where $a_j\in k$ is such that $a_j^p=c_{i+1}^{-1}c_j$. Note that $k$ is perfect so such an $a_j$ exists. 

   We shall verify that these $\beta_{i+1,j}$ satisfy the four assertions of the claim. Firstly, since $L_{i+1}=L_i(\beta_{i,i+1})$, for $i+1<n\le l-1$ we have 
   $$L_{i+1}(\beta_{i+1,j}; i+1 <j\le n)=L_i(\beta_{i,j}; i < j\le n)=L_n$$ 
   Hence the third property is satisfied.
 
   We Compute
   \begin{align*}
    \beta_{i+1,j}^p-\beta_{i+1,j}&=\beta_{i,j}^p-\beta_{i,j}-a_j^p\beta_{i,i+1}^p+a_j\beta_{i,i+1}\\
       &=c_jy_i^{-1}+d_j-a_j^p(\beta_{i,i+1}+c_{i+1}y_i^{-1}+d_{i+1})+a_j\beta_{i,i+1}\\
       &=(c_j-a_j^pc_{i+1})y_i^{-1}+d_j-a^pd_{i+1} +(a_j-a_j^p)\beta_{i,i+1}\\
       &=(a_j-a_j^p)\beta_{i,i+1}+d_j-a_j^pd_{i+1}
   \end{align*}
  Hence $\beta_{i+1,j}^p-\beta_{i+1,j}\in L_{i+1}$. If $a_j=a_j^p$ then $\beta_{i+1,j}^p-\beta_{i+1,j}\in k$ but this will lead to a residue field extension for $S/R$ which contradicts the assumption that $S/R$ is totally ramified. Hence $a_j\ne a_j^p$ and
  \begin{equation}
    \beta_{i+1,j}^p-\beta_{i+1,j}=(\text{nonzero constant})\beta_{i,i+1}+\text{constant}    
  \end{equation}
  So $v_{i+1}(\beta_{i+1,j}^p-\beta_{i+1,j})=v_{i+1}(\beta_{i,i+1})=-1$. We have now verified the first two properties of the claim too.

  Finally, $v_{i+2}(\beta_{i+1,i+2}^p)=v_{i+2}(\beta_{i+1,i+2}^p-\beta_{i+1,i+2})=v_{i+2}(\beta_{i,i+1})$. So we deduce that $v_{i+2}(\beta_{i+1,i+2})=p^{-1}v_{i+2}(\beta_{i,i+1})=p^{-1}pv_{i+1}(\beta_{i,i+1})=-1$. This completes the proof of the claim.
 \end{proof}

 The field extension $L_{l-1}/\QF(R)$ is Galois with Galois group $(\ZZ/p\ZZ)^{l-1}$ and $\Gal(\QF(S)/L_{l-1})=\ZZ/p\ZZ$. Moreover, both $T_{l-1}/R$ and $S/T_{l-1}$ are totally ramified extension. Note that $L_{l-1}=\QF(R)(\alpha_1,\ldots,\alpha_{l-1})$. So by induction hypothesis $d_{T_{l-1}/R}=2p^{l-1}-2$. 

 Since $\QF(S)=L_{l-1}(\beta_{l-1,l})$, $\beta_{l-1,l}^p-\beta_{l-1,l}\in L_{l-1}$ and $v_{l-1}(\beta_{l-1,l}^p-\beta_{l-1,l})=-1$, we have $d_{S/T_{l-1}}=2p-2$ by ``$l=1$ case''. 

 Finally using the transitivity of different, we see that $d_{S/R}=e_{S/T_{l-1}}d_{T_{l-1}/R}+d_{S/T_{l-1}}=p(2p^{l-1}-2)+2p-2= 2p^l-2$. This completes the proof of the lemma.
\end{proof}

\begin{pro}\label{filtration.Galois}
 Let $i\ge 1$ and $S/R$ be a finite extension of complete DVRs over a perfect field $k$ of characteristic $p$ such that $\Gal(\QF(S)/\QF(R))=G=G_i$. Let $L$ be the subfield of $\QF(S)$ generated over $\QF(R)$ by all $\alpha\in \QF(S)$ such that $v_R(\alpha^p-\alpha)=-i$. Then $G_{i+1}\supset \Gal(\QF(S)/L)$. 
\end{pro}

\begin{proof}
  Let $L'=\QF(S)^{G_{i+1}}$ and $H=\Gal(\QF(S)/L)\le G$. Let $T$ and $T'$ be the normalization of $R$ in $L$ and $L'$ respectively. Since $G_{i+1}$ is a normal subgroup of $G$, the extension $L'/\QF(R)$ is Galois and $\Gal(L'/\QF(R))=G/G_{i+1} (=\bar G$ say). Moreover the ramification filtration on $\bar G$ is given by $\bar G_i=\bar G$ and $\bar G_{i+1}=\{e\}$ (Proposition \ref{filtration}). If $G_{i+1}=G$ then $H\subset G_{i+1}$ and we are done. So we may assume $G_{i+1}\ne G$. By Proposition \ref{higher.ramification.group} $\bar G\ne \{e\}$ is isomorphic to the direct sum of copies of $\ZZ/p\ZZ$. 

  Let $L''\subset L'$ be any $\ZZ/p\ZZ$-extension of $\QF(R)$. By Artin-Schrier theory there exists $\alpha\in L''\setminus \QF(R)$ such that $\beta:=\alpha^p-\alpha \in \QF(R)$. Let $x$ be a local parameter of $R$ then $R=k[[x]]$. If $v_R(\beta)>0$ then $\alpha=c-\beta-\beta^p-\beta^{p^2}-\ldots \in R$ for some $c\in \FF_p$. So $v_R(\beta)\le 0$. Moreover since $G_0=G$, $S/R$ is totally ramified. So $v_R(\beta)\ne 0$ and hence $v_R(\beta)\le 0$. If $v_R(\beta)$ is a multiple of $p$ then $\beta=c_0x^{pl}+c_1x^{pl+1}+\ldots$, for some integer $l< 0$. Let $c\in k$ be such that $c^p=c_0$ and let $\alpha'=\alpha-cx^l$. Then $\beta':=\alpha'^p-\alpha'=\beta-c_0x^{pl}+cx^l$, $v_R(\beta')>v_R(\beta)$ and $L''=\QF(R)(\alpha)=\QF(R)(\alpha')$. Hence by such modifications we may assume $v_R(\alpha^p-\alpha)=-r<0$ is coprime to $p$. Let $T''$ be the normalization of $R$ in $L''$. By explicit calculation of the different and using Hilbert's different formula, the degree of the different $d_{T''/R}=(r+1)(p-1)$. Since $\bar G_{i+1}$ is trivial and $\bar G_i=\bar G$, by Hilbert's different formula $d_{T'/R}=(i+1)|\bar G|-i-1$. Let $\bar H$ be the index $p$ subgroup of $\bar G$ such that $L''=L^{\bar H}$. Then the ramification filtration on $\bar H$ (coming from the extension $T'/T''$) is induced from $\bar G$. Hence $d_{T'/T''}=(i+1)|\bar H|-i-1$. Using Lemma \ref{Hilbert-transitivity} and $e_{T'/T''}=|\bar H|$, we obtain
  $$ (i+1)|\bar G|-i-1= |\bar H|(r+1)(p-1)+(i+1)|\bar H|-i-1$$
  Using $|\bar G|=p|\bar H|$ above and solving for $r$, one gets $r=i$. Hence $L''\subset L$. Since $L''$ was an arbitrary $\ZZ/p\ZZ$-extension of $\QF(R)$ contained in $L'$ and $L'$ is generated by such $\ZZ/p\ZZ$-extensions, we have that $L'\subset L$. So by the fundamental theorem of Galois theory $H\subset G_2$.
\end{proof}

\begin{pro}\label{characterize.G_2}
 Let $S/R$ be a finite extension of complete DVRs over a perfect field $k$ of characteristic $p$ such that $\Gal(\QF(S)/\QF(R))=G=G_1$. Let $L$ be the subfield of $\QF(S)$ generated over $\QF(R)$ by all $\alpha\in \QF(S)$ such that $v_R(\alpha^p-\alpha)=-1$. Then $G_2=\Gal(\QF(S)/L)$. 
\end{pro}
\begin{proof}
  In view of Proposition \ref{filtration.Galois}, it is enough to show $G_2\subset H:=\Gal(\QF(S)/L)$. Let $T$ be the normalization of $R$ in $L$. Note that $L/\QF(R)$ is a Galois extension with Galois group $G/H$. By Lemma \ref{lemma.different} $d_{T/R}=2|G/H|-2$. So using Lemma \ref{Hilbert-transitivity} one gets:
  $$2|G|-2+\sum_{i=2}^{\infty} (|G_i|-1)=|H|(2|G/H|-2)+2|H|-2+\sum_{i=2}^{\infty}(|H_i|-1)$$
  Rearranging and using $|G|=|G/H|\cdot|H|$, the above reduces to the following
  $$2|G/H|-2+|H|^{-1}\sum_{i=2}^{\infty} (|G_i|-|H_i|)=2|G/H|-2$$
  So $G_i=H_i$ for $i\ge 2$. Hence $G_2=H\cap G_2$ which implies $G_2\subset H$.
\end{proof}

\begin{cor}
 Let $S/R$ be a finite extension of complete DVRs over a perfect field $k$ of characteristic $p$ such that $\Gal(\QF(S)/\QF(R))=G=F^1G$. Then $F^2G\ne G$ iff there exists $\alpha\in \QF(S)$ such that $\alpha^p-\alpha\in \QF(R)$ and $v_R(\alpha^p-\alpha)=-1$.
\end{cor}

\section{Reducing Inertia}

For a local ring $R$, let $m_R$ denote the maximal ideal of $R$. In this section we shall show how even the wild part of inertia subgroup of a Galois cover can be reduced. We begin with the following lemma.
\begin{lemma}\label{completelocal}
 Let $R$ be a DVR and $K$ be the quotient field of $R$. Let $L$ and $M$ be finite separable extensions of $K$ and $\Omega=LM$ their compositum. Let $A$ be a DVR dominating $R$ with quotient field $\Omega$. Note that $S=A\cap L$ and $T=A\cap M$ are DVRs. Let $\hat K$, $\hat L$, $\hat M$ and $\hat \Omega$ be the quotient field of the complete DVRs $\hat R$, $\hat S$, $\hat T$ and $\hat A$ respectively. If $A/m_A=S/m_S$ then $\hat \Omega = \hat L \hat M$.  Here all fields are viewed as subfields of an algebraic closure of $\hat K$.
\end{lemma}

\begin{proof}
 Note that $\hat L$ and $\hat M$ are contained in $\hat \Omega$. So $\hat L\hat M \subset \hat \Omega$. Let $\pi_A$ denote a uniformizing parameter of $A$. Then $\pi_A \in LM \subset \hat L \hat M$. So it is enough to show that $\hat \Omega =\hat L[\pi_A]$. Note that $\hat S[\pi_A]$ is a finite $\hat S$-module, hence it is a complete DVR \cite{cohen}. Also $\hat S \subset \hat S[\pi_A] \subset \hat A$ and $\pi_A$ generate the maximal ideal of $\hat A$, hence $\pi_AS$ is the maximal ideal of $\hat S[\pi_A]$. Moreover, the residue field of $\hat S$ is equal to $S/m_S=A/m_A$ which is same as the residue field of $\hat A$. Hence the residue field of $\hat S[\pi_A]$ is also same as the residue field of $\hat A$. So $\hat S[\pi_A]=\hat A$ (by \cite[Lemma 4]{cohen}). Hence the quotient field of $\hat S[\pi_A]$ is $\hat \Omega$. But that means $\hat L[\pi_A]=\hat \Omega$.
\end{proof}

\begin{cor}
 Let the notation be as in the above theorem. If $\hat L\subset \hat M$ then $A/T$ is an unramified extension. 
\end{cor}

\begin{proof}
 Since $\Omega/M$ is finite extension, so is $\hat \Omega/\hat M$. Hence $\hat A$ is a finite $\hat T$-module. By the above lemma and the hypothesis $\hat \Omega=\hat M$. So $\hat A=\hat T$, i.e. $A/T$ is unramified.
\end{proof}

Let $k$ be any field.
\begin{thm}\label{thm-inertia}
 Let $X\to Y$ and $Z\to Y$ be Galois covers of regular $k$-curves branched at $\tau \in Y$. Let $\tau_x$ and $\tau_z$ be closed points of $X$ and $Z$ respectively, lying above $\tau$. Suppose $k(\tau_z)=k(\tau)$. Let $W$ be an irreducible dominating component of the normalization of $X\times_Y Z$ containing the closed point $(\tau_x,\tau_z)$. Then $W\to Y$ is a Galois cover ramified at $\tau$ and the decomposition subgroup of the cover at $\tau$ is the Galois group of the field extension $QF(\hat{\cO}_{X,\tau_x})QF(\hat{\cO}_{Z,\tau_z})/QF(\hat{\cO}_{Y,\tau_y})$.
\end{thm}
\begin{proof}
 Let $R=\cO_{Y,\tau}$. Note that $R$ is a DVR. Let  $K$ be the quotient field of $R$. Let $L$ and $M$ be the function field of $X$ and $Z$ respectively and $\Omega=LM$ be their compositum. By definition $W$ is an irreducible regular curve with function field $\Omega$ and the two projections give the covering morphisms to $X$ and $Y$. Let $\tau_w$ denote the closed point $(\tau_x,\tau_z) \in W$ and $A=\cO_{W,\tau_w}$. Since $\tau_w$ lies above $\tau_x$ under the covering $W\to X$ and above $\tau_z$ under the covering $W\to Z$, we have that $A\cap L=\cO_{X,\tau_x}(=S\text{ say})$ and $A\cap M=\cO_{Z,\tau_z}(=T\text{ say})$. Since $k(\tau_z)=k(\tau)$ and $k(W)=k(X)k(Z)$ we get that $k(\tau_w)=k(\tau_z)k(\tau_x)=k(\tau_x)$. But this is same as $A/m_A=S/m_S$. So using the above lemma, we conclude that $\hat L\hat M=\hat \Omega$.

 The decomposition group of the cover $W\to Y$ at $\tau_w$ is given by the Galois group of the field extension $\hat \Omega/\hat K$ (\cite[Corollary 4, Section 8.6, Chapter 6]{bourbaki}). This completes the proof because $\hat \Omega= \hat L\hat M=QF(\hat{\cO}_{X,\tau_x})QF(\hat{\cO}_{Z,\tau_z})$ and $\hat K=QF(\hat{\cO}_{Y,\tau})$.
\end{proof}
 
\begin{pro}\label{killing.wild}
 Let $\Phi:X\to Y$ be a $G$-cover of regular $k$-curves ramified at $\tau_x\in X$ and let $\tau=\Phi(\tau_x)$. Let $G_{\tau}$ and $I_{\tau}$ be the decomposition subgroup and the inertia subgroup respectively at $\tau_x$. Let $N\le I_{\tau}$ be a normal subgroup of $G_{\tau}$. Suppose there exist a Galois cover $\Psi:Z\to Y$ of regular $k$-curves ramified at $\tau_z\in Z$ with $\Psi(\tau_z)=\tau$ such that $k(\tau_z)=k(\tau)$ and the fixed field $\QF(\hat \cO_{X,\tau_x})^N$ is same as the compositum $\QF(\hat \cO_{Z,\tau_z})k(\tau_x)$. Let $W$ be an irreducible dominating component of the normalization of $X\times_Y Z$ containing $(\tau_x,\tau_z)$. Then the natural morphism $W \to Z$ is a Galois cover. The inertia group and the decomposition group at the point $(\tau_x,\tau_z)$ are $N$ and an extension of $N$ by $\Gal(k(\tau_x)/k(\tau))$ respectively. 
\end{pro}

\begin{proof}
 Let $\tau_w\in W$ be the point $(\tau_x,\tau_z)$. Applying Theorem \ref{thm-inertia}, we obtain that the decomposition group of the Galois cover $W\to Y$ at $\tau_w$ is isomorphic to $G_{\tau_w}=\Gal(\QF(\hat{\cO}_{X,\tau_x})\QF(\hat{\cO}_{Z,\tau_z})/\QF(\hat{\cO}_{Y,\tau}))$. Since $\QF(\hat \cO_{Z,\tau_z})\subset \QF(\hat \cO_{X,\tau_x})$, we have $G_{\tau_w}=G_{\tau}=\Gal(\QF(\hat{\cO}_{X,\tau_x})/\QF(\hat{\cO}_{Y,\tau}))$. Since $k(\tau_z)=k(\tau)$, the inertia group and the decomposition group of the cover $Z\to Y$ at $\tau_z$ are both $\Gal(\QF(\hat \cO_{Z,\tau_z})/\QF(\hat \cO_{Y,\tau}))$. Since $\QF(\hat \cO_{X,\tau_x})^N=\QF(\hat \cO_{Z,\tau_z})k(\tau_x)$ we also obtain that $\Gal(\QF(\hat \cO_{Z,\tau_z})k(\tau_x)/\QF(\hat \cO_{Y,\tau}))=G_{\tau}/N$. Moreover, we have $G_{\tau}/I_{\tau}=\Gal(k(\tau_x)/k(\tau))=\Gal(k(\tau_x)\QF(\hat\cO_{Y,\tau})/\QF(\hat \cO_{Y,\tau}))$. Since $\hat \cO_{Z,\tau_z}/\hat \cO_{Y,\tau}$ is totally ramified, $\QF(\hat \cO_{Z,\tau_z})$, $k(\tau_x)\QF(\hat \cO_{Y,\tau})$ are linearly disjoint over $\QF(\hat \cO_{Y,\tau})$. 
 \[
  \xymatrix{
     & \QF(\hat\cO_{X,\tau_x})\\
     & \QF(\hat\cO_{Z,\tau_z})k(\tau_x)\ar[u]^{N}\\
   \QF(\hat\cO_{Z,\tau_z})\ar[ru]^{G_{\tau}/I_{\tau}} & &  \QF(\hat\cO_{Y,\tau})k(\tau_x)\ar[lu]\\
     & \QF(\hat\cO_{Y,\tau})\ar[uu]^{G_{\tau}/N}\ar[ru]_{G_{\tau}/I_{\tau}}\ar[lu]
   }
 \]

 So $\Gal(\QF(\hat \cO_{Z,\tau_z})k(\tau_x)/\QF(\hat \cO_{Z,\tau_z}))=\Gal(k(\tau_x)/k(\tau))$. So the decomposition group of $W \to Z$ is $\Gal(\QF(\hat \cO_{X,\tau})/\QF(\hat \cO_{Z,\tau_z}))$ which is an extension of $N$ by $\Gal(k(\tau_x)/k(\tau))$ and the inertia group is $\Gal(\QF(\hat \cO_{X,\tau})/\QF(\hat \cO_{Z,\tau_z})k(\tau_x))=N$.
\end{proof}

Let $k$ be an algebraically closed field of characteristic $p>0$.
\begin{thm}\label{killing-inertia}
 Let $\Phi:X\to Y$ be a $G$-Galois cover of regular $k$-curves. Let $\tau_x\in X$ be a ramification point and $\tau=\Phi(\tau_x)$. Let $I$ be the inertia group of $\Phi$ at $\tau_x$.  There exists a cover $\Psi:Z\to Y$ of deg $|I|$, such that the cover $W\to Z$ is \'etale over $\tau_z$ where $W$ is the normalization of $X\times_Y Z$ and $\tau_z\in Z$ is such that $\Psi(\tau_z)=\tau$. Moreover if there are no non-trivial homomorphism from $G \to P$ where $P$ is a $p$-sylow subgroup of $I$ then $W \to Z$ is a $G$-cover of irreducible regular $k$-curves.
\end{thm}

\begin{proof}
 Since $I$ is the inertia group, it is isomorphic to $P\rtimes \mu_n$ where $(p,n)=1$ and $\mu_n$ is a cyclic group of order $n$. Let $y$ be a local coordinate of $Y$ at $\tau$ such that $k(Y)[y^{1/n}]\cap k(X)=k(Y)$. Let $Z_1$ be the normalization of $Y$ in $k(Y)[y^{1/n}]$. Then $Z_1\to Y$ is a $\mu_n$-cover branched at $\tau$ such that $k(Z_1)$ and $k(X)$ are linearly disjoint over $k(Y)$. Let $\tau_{z1}\in Z_1$ be a point lying above $\tau$. Let $X_1$ be the normalization of $X\times_Y Z_1$. Then by the above theorem $\Phi_1:X_1\to Z_1$ is a $G$-cover of irreducible regular $k$-curves and the inertia group at $(\tau_x,\tau_{z1})$ is $P$.

 Let $Y_1=Z_1$, $\tau_{x1}=(\tau_x,\tau_{z1})$ and $\tau_1=\tau_{z1}$. Then $\Phi_1:X_1\to Y_1$ is a $G$-cover with $\Phi_1(\tau_{x1})=\tau_1$ and the inertia group of this cover at $\tau_{x1}$ is $P$. Let $y_1$ be a regular parameter of $Y_1$ at $\tau_1$. Then $k(Y_1)/k(y_1)$ is a finite extension. Since $Y_1$ is a regular curve, we get a finite morphism $\alpha:Y_1\to \PP^1_{y_1}$ such that $\alpha(\tau_1)$ is the point $y_1=0$ and $\alpha$ is \'etale at $\tau_1$ (as $\hat \cO_{Y_1,\tau_1}=k[[y_1]]$).

 Note that $\QF(\hat \cO_{X,\tau_{x1}})/k((y_1))$ is a $P$-extension. By \cite[Cor 2.4]{Harbater-moduliofpcovers}, there exist a $P$-cover $V\to \PP^1_{y_1}$ branched only at $y_1=0$ (where it is totally ramified) such that $\QF(\hat{\cO}_{V,\theta})=\QF(\hat{\cO}_{X_1,\tau_{x1}})$ as extensions of $k((y_1))$. Here $\theta$ is the unique point in $V$ lying above $y_1=0$. Since $V\to \PP^1_{y_1}$ is totally ramified over $y_1=0$ and $Y_1\to \PP^1_{y_1}$ is \'etale over $y_1=0$, the two covers are linearly disjoint. Let $Z$ be the normalization of $V\times_{\PP^1_{y_1}} Y_1$. Then the projection map $Z\to Y_1$ is a $P$-cover. Let $\tau_z\in Z$ be the closed point $(\theta,\tau_1)$. By Lemma \ref{completelocal}, $\QF(\hat \cO_{Z,\tau_z})=\QF(\hat \cO_{V,\theta})\QF(\hat \cO_{Y_1,\tau_1})=\QF(\hat \cO_{X_1,\tau_{x1}})$. Applying Proposition \ref{killing.wild} with $N=\{e\}$, we get that an irreducible dominating component $W$ of the normalization of $X_1 \times_{Y_1} Z$ is a Galois cover of $Z$ such that the inertia group over $\tau_z$ is $\{e\}$. Hence the normalization of $X_1 \times_{Y_1} Z$ is a cover of $Z$ \'etale over $\tau_z$.

 Moreover, there are no nontrivial homomorphism from $G$ to $P$ implies that $k(Z)$ and $k(X_1)$ are linearly disjoint over $k(Y_1)$. Hence $W \to Z$ is a $G$-cover. We take $Z\to Y$ to be the composition $Z\to Y_1\to Y$. Note that the morphism $X\times_Y Z\to Z$ is same as $X_1 \times_{Y_1} Z \to Z$ and the degree of the morphism $Z\to Y$ is $|P|n=|I|$. 
\end{proof}

\begin{thm}\label{reducing-inertia}
 Let $\Phi:X\to \PP^1$ be a $G$-Galois cover of regular $k$-curves. Suppose $\Phi$ is branched only at one point $\infty\in \PP^1$ and the inertia group of $\Phi$ over $\infty$ is  $I$. Let $P$ be a subgroup of $I$ such that $I_1 \supset P \supset I_2$. Suppose there are no nontrivial homomorphism from $G$ to $P$. Then there exist a $G$-cover $W\to \PP^1$ ramified only at $\infty$ and the inertia group at $\infty$ is $P$.
\end{thm}

\begin{proof}
 Let $n=[I:I_1]$ be the tame ramification index of $\Phi$ at $\infty$. Let $x$ be a local coordinate on $\PP^1$ and the point $\infty$ is $x=\infty$. Let $\PP^1_y\to \PP^1_x$ be the Kummer cover obtained by sending $y^n$ to $x$. Since $\Phi$ is \'etale at $x=0$ and the cover $\PP^1_y\to \PP^1_x$ is totally ramified at $x=0$ the two covers are linearly disjoint. So letting $W$ to be the normalization of $X\times_{\PP^1_x} \PP^1_y$, we obtain a $G$-cover $\Phi_1:W\to \PP^1_y$ of regular $k$-curves. Moreover by Abhyankar's lemma $\Phi_1$ is ramified only at $y=\infty$ and the inertia group of $\Phi_1$ at $y=\infty$ is same the subgroup $I_1$ of $I$. So replacing $\Phi$ by $\Phi_1$, we may assume $I=I_1$. Also since $I_1/I_2$ is abelian, $P$ is a normal subgroup of $I$.

 Let $\tau\in X$ be a point above $x=\infty$. Let $S=\hat\cO_{X,\tau}$ and $R=\hat\cO_{\PP^1,\infty}$ then $R=k[[x^{-1}]]$ and $\Gal(QF(S)/\QF(R))=I$. Let $L=\QF(S)^P$. Then by Proposition \ref{characterize.G_2}, $L=\QF(R)(\alpha_1,\ldots,\alpha_l)$ where $\alpha_i\in \QF(S)$ is such that $v_R(\alpha_i^p-\alpha_i)=-1$ for $1\le i \le l$. Let $T$ be the normalization of $R$ in $L$. Then $\spec(T)$ is a principal $P$-cover of $\spec(R)$. By \cite[Corollary 2.4]{Harbater-moduliofpcovers}, this extends to a $P$-cover $\Psi:Z\to \PP^1_x$ ramified only at $x=\infty$ where it is totally ramified. Let $\tau_z\in Z$ be the point lying above $x=\infty$ then $\QF(\hat\cO_{Z,\tau_z})=L=\QF(S)^P$. By Lemma \ref{lemma.different} $d_{T/R}=2|P|-2$. So by Riemann-Hurwitz formula, the genus of $Z$ is given by $$2g_Z-2=|P|(0-2)+d_{T/R}$$
 Hence $g_Z=0$. So $Z$ is isomorphic to $\PP^1$.

 Since there are no nontrivial homomorphism from $G$ to $P$, $\Phi$ and $\Psi$ are linearly disjoint covers of $\PP^1_x$. Let $W$ be the normalization of $X\times_{\PP^1_x} Z$. Now we are in the situation of Proposition \ref{killing.wild}. Hence the $G$-cover $W\to Z$ is ramified only at $\tau_z$ and the inertia group at $\tau_z$ is $P$. This completes the proof as $Z$ is isomorphic to $\PP^1$.
\end{proof}

\begin{rmk}
 Note that if $G$ is a simple group different from $\ZZ/p\ZZ$ then there are no nontrivial homomorphism from $G$ to $P$. Hence the above results apply in this scenario.
\end{rmk}

\begin{cor}
 Let $\Phi:X\to \PP^1$ be a $G$-Galois cover of regular $k$-curves branched only at one point $\infty\in \PP^1$ and the inertia group of $\Phi$ over $\infty$ is  $I$. Suppose there are no nontrivial homomorphism from $G$ to $I_2$. Then conjugates of $I_2$ generate $G$.
\end{cor}

\begin{proof}
 Applying the above theorem with $P=I_2$, we get an \'etale $G$-cover of $\Aff^1$ with the inertia group $I_2$ at $\infty$. Hence the conjugates of $I_2$ generate $G$ since a nontrivial \'etale cover of $\Aff^1$ must be wildly ramified over $\infty$.
\end{proof}

\end{document}